\documentclass[12pt, reqno]{amsart}
\usepackage{amsmath, amsthm, amscd, amsfonts, amssymb, graphicx, color}
\usepackage[bookmarksnumbered, colorlinks, plainpages]{hyperref}

\newtheorem{theorem}{Theorem}[section]
\newtheorem{lemma}[theorem]{Lemma}
\newtheorem{proposition}[theorem]{Proposition}
\newtheorem{corollary}[theorem]{Corollary}
\theoremstyle{definition}
\newtheorem{definition}[theorem]{Definition}

\theoremstyle{remark}

\numberwithin{equation}{section}

\begin{document}
\setcounter{page}{1}

\title[CHEBYSHEV SET PROBLEM IN NORMED LINEAR SPACES]{A NOTE ON THE CHEBYSHEV SET PROBLEM IN NORMED LINEAR SPACES}
\author[Samson Owiti, Benard Okelo, Julia Owino]{Samson Owiti, Benard Okelo$^{*}$, Julia Owino}

\address{Department of Pure and Applied Mathematics\\ School of Biological, Physical, Mathematics and Actuarial Sciences\\ Jaramogi Oginga Odinga  University of Science and Technology\\ Box 210-40601, Bondo-Kenya.}
\email{bnyaare@yahoo.com}



\date{Received: 13 January, 2023.
\newline \indent $^{*}$ Corresponding author}

\begin{abstract}
Best approximation (BA) is  an interesting field in functional analysis that has attracted a lot of attention from many researchers for a very long period of time up-to-date. Of greatest consideration is the characterization of the Chebyshev set (CS) which is a subset of a normed linear space (NLS) which contains unique BAs. However, a fundamental question remains unsolved to-date regarding the convexity of the CS in infinite NLS known as the CS problem. The question which has not been answered is: Is every CS in a NLS convex?.  This question has not got any solution including  the simplest form of a real Hilbert space (HS). In this note, we characterize CSs and convexity  in NLSs. In particular, we consider the space of all real-valued norm-attainable functions. We show that CSs of the space of all real-valued norm-attainable functions are convex when they are closed, rotund and admits both Gateaux and Fr\'{e}chet differentiability conditions.
\end{abstract} \maketitle

\section{Introduction}
Studies in approximation theory have been carried out by many mathematicians over decades (see \cite{Asp}, \cite{Bac}, \cite{Hal} and \cite{Phel} and the references therein). The most important  basic question in the field of BA is the concern about the existence of
BAs \cite{Fle}. This is because BA theory has several applications involving finding solution to  systems of equations \cite{Hal}. This work is useful in contributing  knowledge in functional analysis by providing at least a partial solution to the CS problem \cite{Oke3}. It will also be useful in solving convex optimization problems and finding solutions to differential equations \cite{Oke1}. Best approximation (BA) is  an interesting field in functional analysis that has attracted a lot of attention from many researchers for a very long period of time up-to-date (see  \cite{Bal1} and \cite{Bal2}, \cite{Man} and the references there in). Of greatest consideration is the characterization of the CS which is a subset of a NLS which contains unique BAs \cite{Sang}.  Approximation theory involves obtaining best approximation of functions using simple functions whether they are linear or nonlinear \cite{Bal2}, \cite{Fer} and \cite{Nar}.
Martin \cite{Mar} characterized remotality of sets with regard to normed linear spaces and in particular for convex sets in Banach spaces. However, the convexity of the Banach spaces was not done in general due to the complex nature of spaces. Mazaheri \cite{Maz3} also considered weakly-Chebyshev subspaces for NLSs  of Banach spaces and did their characterization in terms of nearest and farthest points via distance functions. However, the authors could not give a particular best approximation for the CSs and their convexity even for the simplest case \cite{Maz2} of a HS of $\emph{l}^{2}.$
Zalinescu \cite{Zal} on the study of convex sets and their characterizations  in general spaces determined optimization criteria for vector spaces and left open a question regarding convexity of these spaces.
More recently, Mazaheri and  Salehi \cite{Maz3} studied CSs and considered conditions under which they are convex. However, they could not determine convexity of the CSs in NLSs even for the  simplest set up of HSs.
It is worth noting that various techniques have been used in trying to get a solution to the CS problem \cite{Bal2}. The first technique we discuss in this work is the Bunt-Motzkin Theorem. This theorem is a result on the converse of the CS problem. It asserts that if a set $C$ is Chebyshev then it implies that it is convex in Bergman spaces \cite{Bor}. However, this assertion still remains unknown if it holds for infinite dimensional Hilbert spaces.
The other technique is the Fr$\grave{e}$chet differentiability \cite{Fra}.
This is a derivative defined mostly in normed spaces. Fr$\grave{e}$chet differentiability occurs on real-valued functions or vector valued functions of multiple variables. It is applicable mostly and particularly on directional derivative where the continuity of the map is essential \cite{Oke4}.
 Gateaux differentiability conditions are also very instrumental conditions in BAs. A Gateaux derivative is a fundamental principle in differential calculus which is a generalization on functions which are continuous on Banach spaces. It is useful in carrying out approximations in locally convex spaces \cite{Oke6}. It is useful in formalization of functional derivatives which are important in best approximations in calculus of variations. Moreover, its useful since it takes care of nonlinear functions also.
Best approximation techniques have also been employed in CS problem \cite{Sang}. These are techniques in approximation theory which useful in obtaining best approximation results for functions in various spaces. They include: Polynomial approximations; Chebyshev approximations; Remez techniques and algorithms; and Pade' approximation techniques for optimal polynomials. Despite the fact that these techniques have been used in trying to solve CS problem, the answer to this problem is still elusive \cite{Riv}.
It is very important to unveil a detailed account of research work which have been done in approximation theory. In particular, we consider literature on conditions under which subsets of NLSs are  Chebyshev. We also discuss the characterizations on distance functions of CSs  in NLSs and finally give a review on  investigations on convexity of CSs in various NLSs \cite{Fer}.
We begin with some brief account on the various studies on CSs and their subsets.
CSs are important sets in approximation theory due to their properties. A lot of characterizations have been done on these sets with interesting results obtained. Their subsets are also interesting as they carry hereditary properties in them. Fletcher and Moors \cite{Fle} characterized CSs and showed that a CS is particularly a subset of a NLS which has properties that helps in the establishment of   best approximations results which are unique. The authors investigated characteristics  of the metric projection and obtained necessary and sufficient conditions under which a a subset of a NLS becomes a CS and also conditions for a CS to satisfies  convexity properties. Moreover, the authors gave an example where they constructed a nonconvex CS  as shown in the next result.
\begin{proposition}\label{prop1}(\cite{Fle}, Proposition 2.12)
Let $W$ be a NLS. A subset of $W$ which is closed is also convex if and
only if it is midpoint convex.
\end{proposition}
\noindent Proposition \ref{prop1} characterizes NLSs in terms of midpoint convexity and shows that a subset of  a NLS can be convex if its closed and satisfies midpoint convexity. However, this result does not indicate whether the space $W$ is convex in general even if it is a CS. It is noted that closedness is useful for midpoint convexity but not for convexity in general.\\
\noindent Vlasov \cite{Vla} in the earlier years characterized normed spaces in terms of approximate properties. Historically, it can be  shown in  brief,  the main contributions of researchers in this field of study. In 1934 Bunt  proved the convexity  of a CS in the real plane. Also 1938 Kritikos followed  Bunt's results and  extended the theorem of Bunt to an $n$-dimensional real space. This was followed by the work of Efimov and Stechkin  in 1961 which showed that a CS of a general HS which is approximately
compact satisfies  convexity condition. To conclude the history, Klee  considered weak closedness and proved that CS which is weakly closed satisfies the convexity condition. This history and more details can be found in Mantegazza \cite{Man}.
\noindent To consider particular cases, the author in \cite{Sang} considered proximal sets which are related to Chebyshev sets and gave the result below.
\begin{theorem}\label{thm1}
Let $W$ be a NLS. Consider $J$ as a proximinal set of  $W$. Then $J$
is nonempty and closed.
\end{theorem}
\noindent Theorem \ref{thm1} considers proximal sets and characterizes closedness and the content of the sets. However, it becomes very difficult to come up with a structure  of the
metric projection function in terms of its geometry. Nonetheless, an exception on this assertion can be considered when $J$ is a
subspace.
 Next we consider distance functions of Chebyshev sets in NLSs. Distance functions are also important in characterizing CSs. By the result of Asplund \cite{Asp}, the CS problem was given a different dimension to consider metric projections as shown in the next proposition.
\begin{proposition}\label{prop2}(\cite{Fle}, Proposition 2.2)
Let $W$ be a NLS. Consider $J$ as a CS of  $W$. Then the distance function for $J$
 satisfies nonexpansivity  and continuity conditions.
\end{proposition}
\noindent Proposition \ref{prop2} describes CSs and their distance functions in terms of continuity and nonexpansivity in a general set up.
\noindent Lastly we consider CSs and convexity. We consider literature on the key question of this study, that is,
Is every CS in HS convex?.
We begin with the following analogy of the result of Borwein \cite{Bor} on closedness, reflexivity and rotundity of Hilbert spaces.
\begin{lemma}\label{lem1}(\cite{Bor}, Fact 3)
All CSs are closed and all
closed sets satisfying convexity condition are Chebyshev in a rotund reflexive space. Particularly,
all nonempty closed  sets satisfying convexity condition in HS are Chebyshev.
\end{lemma}
\noindent Lemma \ref{lem1} gives an elaborate characterization of Chebyshev sets and convexity in terms of reflexivity and rotundity that requires uniqueness property however this does not answer Cs problem in general.\\
In summary, with all these considerations, a fundamental question remains unsolved to-date regarding the convexity of CSs in infinite NLSs known as the CS problem. The question which has not been answered is: Is every CS in a NLS convex?.  This question has not been answered even in the simplest case of a real HS. In this regard, it is worth characterizing CS and convexity  in NLSs. We also attempt to answer this question partially in a particular case of the NLS space of all norm-attainable real-valued functions.
\noindent This work is organized as follows: For the first section, we begin  with a mathematical background as given in this introduction followed with the preliminary basic concepts that helps us to understand this work. We then provide the main results and finally the conclusion.

\section{Preliminaries}
\noindent For a better understanding of this work, we outline the basic definitions that are key to this note on  CS and convexity in NLSs.

\begin{definition}(\cite{Fra})
Let  $W$ be a NLS and $G$ be a nonempty subset of $W$.
Consider the particular point $\zeta \in W$. We define the distance from the point
$\zeta$ to $G$ by $d(\zeta,G) = \inf_{\eta\in G} \|\zeta - \eta\|$, and  the map
$\zeta \mapsto d(\zeta,G)$ is called  the distance function for $G$. We call $\zeta$ the nearest point in $G.$
\end{definition}

\begin{definition}(\cite{Nar})
A set $G$ in a NLS $W$ is called a CS if every point in $W$ has a unique nearest point in $G.$ That is,  CS is a subset of a NLS that admits unique best approximations.
\end{definition}

\begin{definition}(\cite{Fra}) Let $D$ be a nonempty set. A subset $E$ of $D$ is said to be convex if for all $\zeta,\eta\in E$  the line segment connecting $\zeta$ and $\eta$ is in $E$, that is, $(1 - \alpha)\zeta + \alpha\eta$ is in $E$ for $\zeta,\eta\in E$, and $\alpha\in[0, 1].$
\end{definition}
\noindent At this point, we proceed to give the main results of this paper. These results are restricted to the space of all norm-attainable real-valued functions. A function $\phi$ is said to be norm-attainable if there exists a unit vector $\xi$ in the domain of $\phi$ such that $\|\phi(\xi)\|=\|\phi\|.$ The space of all norm-attainable real-valued functions is a NLS. For details on norm-attainability, see \cite{Oke1}-\cite{Oke6} and the references therein.
\section{Main results}
\noindent We provide the main results of this note in this section. We characterize CSs and their subsets and tackle the CS problem. We begin with the following proposition which considers distance functions and Gateaux differentiability. We note that all the spaces and their subspaces are all nontrivial and are strictly NLS spaces of all functions that are norm-attainable unless otherwise stated.

\begin{proposition}
Let $\mathfrak{Q}$ be a NLS space of all norm-attainable real-valued functions and $\mathfrak{J}$ be a closed and smooth subset of $\mathfrak{Q}$. Let
 $\zeta\in \mathfrak{Q}\backslash \mathfrak{J}$ and $\eta$  the nearest point
for $\zeta$ in $\mathfrak{J},$ then  Gateaux differentiability  condition of $\mathfrak{Q}$ holds for
$(\zeta-\eta).$
\end{proposition}\label{propR1}
\begin{proof}
Since the norm of $\mathfrak{Q}$ is Gateaux differentiable from the statement of the proposition, it suffices to prove the exixtence of the unique limit of the the Gateaux derivative $\displaystyle{\lim_{l\rightarrow0}\frac{d_{\mathfrak{J}}(\zeta+l(x-\eta))-d_{\mathfrak{J}}(\zeta)}{l}}.$
From \cite{Oke1}, we deduce that if $l>0 $ then the limit exists. For the uniqueness, we see from the result of \cite{Fle} that the limit of the derivative is unique. It follows then 
that $\langle d'_{\mathfrak{J}}(\zeta),x-\eta\rangle$ holds for $d_{\mathfrak{J}}(\zeta)$. This completes the proof.
\end{proof}
\noindent Proposition 3.1 leads to the interesting question as to what happens when $\mathfrak{Q}$ is rotund. We see this in the next lemma.
\begin{lemma}\label{lemma}
Let $\mathfrak{Q}$ be a NLS space of all norm-attainable real-valued functions and $\mathfrak{J}$ be a closed and smooth CS of $\mathfrak{Q}$. Let
 $\zeta\in \mathfrak{Q}\backslash \mathfrak{J}$ and $\partial d_{\mathfrak{J}}(\zeta)$ be a  singleton set. Then the following conditions hold if the first dual of $\mathfrak{Q}$ is rotund:\\
({i}). $\phi$ on $\mathfrak{J}$ is uniformly continuous.\\({ii}). $\phi$ on $\mathfrak{J}$ is totally bounded.\\({iii}). $\mathfrak{J}$ satisfies  convexity condition.\\ ({iv}). $d_{\mathfrak{J}}$ satisfies  convexity condition.\\ ({v}). $d_{\mathfrak{J}}$ satisfies Gateaux differentiability at $\zeta$.
\end{lemma}
\begin{proof}
We proceed with the proof as follows:\\
Case $(i).$ $\phi$ on $\mathfrak{J}$ is uniformly continuous since every space of norm-attainable functions contains continuous functions.\\
Case $(ii).$ $\phi$ on $\mathfrak{J}$ is totally bounded follows immediately from case $(i).$\\
Case $(iii).$ $\mathfrak{J}$ satisfying  convexity condition follows immediately from the  conditions of the statement of the lemma.\\
Case $(iv).$ $d_{\mathfrak{J}}$ satisfying  convexity condition follows from the fact that $\mathfrak{J}$ satisfies  convexity condition.\\
Case $(v).$
Since $d_{\mathfrak{J}}$ satisfies  convexity condition  and is uniformly continuous at $\zeta$ and from Proposition 3.1
$\partial d_{K}(\zeta)$ is a singleton set, $d_{\mathfrak{J}}$ satisfies Gateaux
differentiability at point $\zeta$ and we attain equality of $d'_{\mathfrak{J}}(\zeta)$ and $\partial d_{\mathfrak{J}}(\zeta)$. This completes the proof.
\end{proof}
\noindent At this point, we state the main theorem of our work that characterizes convexity of $\mathfrak{Q}$ in terms of Fr\'{e}chet differentiability condition.
\begin{theorem}\label{theorem}
Let $\mathfrak{Q}$ be a NLS space of all norm-attainable real-valued functions and $\mathfrak{J}$ be a closed and smooth CS of $\mathfrak{Q}$. Let
 $\zeta\in \mathfrak{Q}\backslash \mathfrak{J}$ and $\partial d_{\mathfrak{J}}(\zeta)$ be a  singleton set. Then $d_{\mathfrak{J}}$ satisfies Fr\'{e}chet differentiability condition at $\zeta$.
\end{theorem}
\begin{proof}
It is known from \cite{Vla} that  the  norm of $\mathfrak{Q}$  and hence the dual norm of $\mathfrak{Q}^{*}$ satisfies Fr\'{e}chet differentiability condition. Also from Lemma \ref{lemma} it implies that $\mathfrak{Q}$ is strictly reflexive. Moreover, $\mathfrak{Q}\backslash \mathfrak{J}$ has the nearest point $\zeta$ and so $\mathfrak{J}$ of $\mathfrak{Q}$  satisfies Fr\'{e}chet differentiability condition and so is $\mathfrak{Q}$.
\end{proof}
\noindent As consequences of Theorem \ref{theorem}, we state the following corollaries.
\begin{corollary}
Every distance function of a CS of the NLS space of all norm-attainable real-valued functions is Fr\'{e}chet differentiable.
\end{corollary}
\begin{proof}
Follows from the conditions of Lemma \ref{lemma} and Theorem \ref{theorem}. The rest is clear from the fact that every rotund CS is convex.
\end{proof}
\begin{corollary}\label{cor1}
Every distance function of a CS of the NLS space of all norm-attainable real-valued functions is Gateaux differentiable.
\end{corollary}
\begin{proof}
Follows immediately from  the conditions of Theorem \ref{theorem} and Corollary \ref{cor1} the proof is complete.
\end{proof}
\section{Conclusion}
In conclusion,  a fundamental question that remains unsolved to-date regarding the convexity of the CS in infinite NLS known as the CS problem has been studied in this work. This CS problem which has not been solved  in totality (even in this note) states that: Is every CS in a NLS convex?  This question has not got any solution even in  the simplest form of a real Hilbert space (HS). In this note, we have characterized Chebyshev sets and their convexity  in NLSs. We considered the NLS space of all real-valued norm-attainable functions. We have shown that Chebyshev subsets of the NLS space of all real-valued norm-attainable functions are convex when they are closed, rotund and admits both Gateaux and Fr\'{e}chet differentiability conditions.

\end{document}